\documentclass[11pt]{article}
\usepackage{graphicx, amsmath, amsthm, amsfonts, enumerate, authblk, amssymb, graphics, fullpage, esint, relsize}
\usepackage{color}
\usepackage[mathscr]{euscript}
\newtheorem{theorem}{Theorem}[section]

\newtheorem{definition}{Definition}[section]

\newtheorem{remark}{Remark}

\makeatletter
\def\blfootnote{\xdef\@thefnmark{}\@footnotetext}
\makeatother
\allowdisplaybreaks

\newcommand{\abs}[1]{\left|#1\right|}
\newcommand{\I}{{\rm i}}
\newcommand{\complex}{\mathbb{C}}
\newcommand{\eval}[1]{\upharpoonright_{#1}}
\newcommand{\cc}[1]{\overline{#1}}
\newcommand{\nats}{\mathbb{N}}
\newcommand{\reals}{\mathbb{R}}

\newcommand{\inner}[2]{\left\langle#1,#2\right\rangle}
\newcommand{\norm}[1]{\left\|#1\right\|}
\DeclareMathOperator{\dom}{dom}

\usepackage{color}
\usepackage{empheq}

\usepackage{amsmath, amsfonts,amssymb,amsthm, authblk, enumerate, bigints, fullpage,bbm}
\usepackage{epsfig}
\usepackage{amsmath}
\usepackage{scrlfile}
\usepackage{srcltx}
\usepackage{graphics}
\usepackage{rotating}
\usepackage{remreset}

\def\widehatdntau{\widehat{\partial}^{(\tau)}_n}

\def\hom{\text{\rm hom}}

\def\soft{{\rm soft}}
\def\stiff{{\rm stiff}}

\def\e{\varepsilon}

\begin{document}
\title{\sc Functional model for generalised resolvents and its application to time-dispersive media}

\author[1]{Kirill D. Cherednichenko}
\author[2]{Yulia Yu. Ershova}
\author[3]{Sergey N. Naboko}
\affil[1]{Department of Mathematical Sciences, University of Bath, Claverton Down, Bath, BA2 7AY, UK}
\affil[2]{Department of Mathematics, Texas A\&M University, College Station, TX 77843-3368, USA}
\affil[3]{Department of Mathematics and Mathematical Physics, Physics Faculty, 
	St.\,Petersburg State University, Peterhoff, St.\,Petersburg, Russia}

\maketitle

\begin{abstract}
Motivated by recent results concerning the asymptotic behaviour of  differential operators with highly contrasting coefficients, whose effective descriptions have involved generalised resolvents, we construct the functional model for a typical example of the latter. This provides a spectral representation for the generalised resolvent, which can be utilised for further analysis, in particular the construction of the scattering operator in related wave propagation setups.
\end{abstract}

\par{\raggedleft\slshape In memoriam Sergey Naboko\par}

\section{From resonant composites to generalised resolvents}

Recent advances in the multiscale analysis of differential equations modelling heterogeneous media with high contrast (``high contrast homogenisation") have shown that when the contrast between the material properties of individual components is scaled appropriately with the typical size of heterogeneity (e.g., period in the case of periodic media), the effective description exhibits frequency dispersion (i.e., the dependence of the wavelength on frequency) or, equivalently in the time domain, a memory-type formulation with a convolution kernel, see \cite{Zhikov2000, Zhikov2005, CherKis, CherCoop, ChEK_future, GrandePreuve}. From the physical perspective, it can be viewed as the result of a resonant behaviour of one of the components of such a composite medium, when the typical length-scale of waves (in the case of an unbounded medium) or eigenmodes (in the case of a bounded region) is comparable to the typical size of heterogeneity. 

The need to quantify the above effect for various classes of boundary value problems (BVPs), which ultimately aims at addressing the r\^{o}le of the underlying microscopic resonance in the overall behaviour of a class of physical systems, has also motivated the development of functional analytic frameworks for the analysis of wave scattering and effects of length-scale interactions for parameter-dependent BVP, see \cite{CherKisSilva, CherKisSilva1, CherKisSilva2, CherKisSilva3}. The approach of the latter works was inspired by a treatment of BVP going back to the so-called Birman-Kre\u\i n-Vishik methodology \cite{Birman, Krein1, Krein2, Vishik} and its recent  development by Ryzhov \cite{Ryzh_spec}, rooted in an earlier construction of the functional model of perturbation theory by one of the authors \cite{MR0500225, MR573902}. The theory of boundary triples, which was introduced in  \cite{Gor, DM, Ko1, Koch}, provides a convenient functional analytic framework for the implementation of the ideas introduced by Birman, Kre\u\i n, and Vishik, as shown in a number of parallel recent developments \cite{Grubb_Robin,Grubb_mixed,BehrndtLanger2007,Gesztesy_Mitrea, Ryzh_spec, BMNW2008}; see also the seminal contributions by Calkin \cite{Calkin},
Boutet de Monvel \cite{BdeM}, Birman and Solomyak \cite{Birman_Solomyak}, Grubb \cite{Grubb_book}, and Agranovich \cite{Agranovich}.

In the process of analysing BVP with high contrast using Ryzhov's method, the r\^{o}le of the generalised resolvent obtained by restricting the problem to the ``soft", or resonant, component has been made transparent: this generalised resolvent is the solution operator of a BVP with a constant symbol and a boundary condition dependent on the spectral parameter. 
The passage to the limit as the contrast goes to infinity then naturally leads to a BVP on the soft component with a boundary condition linear in the spectral parameter \cite{CherKisSilva3}. This form of the effective problem is unsurprising from the point of view of the classical compactness argument \cite{Zhikov2005}: the solution gradients (corresponding to, e.g., the strain tensor in elasticity) are forced to vanish on the ``stiff" component, i.e.,  where the material parameter (such as the elastic modulus) is large. 
Notably, problems of this type, where the dependence of the a boundary condition on the spectral parameter is modelled by a general Herglotz function, have also naturally appeared in the analysis of time-dispersive media \cite{Figotin_Schenker_2005, Figotin_Schenker_2007b}, where generalised resolvents feature prominently.   

The operator-theoretic study of generalised resolvents was initiated by Neumark \cite{Naimark1940, Naimark1943}
and further refined by \v{S}traus \cite{Strauss, Strauss1968, Strauss_survey}, who developed an abstract construction of the functional model, in particular applicable to the study of generalised resolvents. 
This provides for an implicit link to the scattering theory for problems with impedance-type boundary conditions,  i.e., those that feature a non-constant function of the spectral parameter $z$ (which represents the square of frequency in the context of wave propagation). In the Sturm-Liouville context, impedance-type problems have been studied by a number of authors, see in particular \cite{Shkalikov_1983, Mennicken_Moeller} and references therein. 

The characteristic function of Livshitz \cite{livshitz} and the spectral form of the functional model for dissipative operators due to Pavlov \cite{MR0510053} are explicitly connected with the scattering theory, see \cite{MR0206711, AdamyanPavlov}. Therefore, it appears reasonable to pose the question of explicit construction of a functional model in the spirit of Pavlov for generalised resolvents \cite{CEKS2022},
 and to study its implications for impedance-type BVP. Furthermore, in relation to the kind of generalised impedance problems that emerge in the context of resonant homogenisation, it seems natural  to also explore appropriate analogues of Pavlov's model of potentials of zero radius with an internal structure \cite{Pavlov_internal_structure, Pavlov_explicitly_solvable}, resulting in an explicit description of a class of generalised resolvents quantifying the interactions between the resonant and non-resonant parts of the medium. To the best of our knowledge, the present work is the first step in implementing the above programme.  



\section{Motivation for the problem to be analysed}

The problems of the type we consider in this paper have recently appeared in a number of seemingly unrelated contexts, ranging from double-porosity homogenisation for scalar and vector PDEs \cite{ChEK_future, ChKVZ} through dimension reduction in thin networks \cite{thin} to quantum graphs \cite{Physics, GrandePreuve}. In the PDE world, a prototypical model is derived in \cite{CherKisSilva3}. 

Consider a smooth bounded domain $\Omega\subset{\mathbb R}^d,$ $d\ge 2,$ a simply connected inclusion $\Omega_-\subset\Omega$ with a $C^{1,1}$ boundary $\Gamma$  located at a positive distance from $\partial\Omega,$ and denote $\Omega_+:=\Omega\setminus\overline{\Omega}_-.$  
Furthermore, consider the space $\widetilde{H}=L^2(\Omega_+)\oplus \mathbb C$ and its linear subset 
\begin{equation}
	\label{eq:domain_fin_modI}
	\dom(A)=\biggl\{
	\binom{u_+}{\beta}\in \widetilde{H}:\ u_+\in H^2(\Omega_+),\ \ u_+|_\Gamma=\frac{\beta}{\sqrt{|\Omega_-|}}{\mathbbm 1}_\Gamma, \ \ \dfrac{\partial u_+}{\partial n_+}\biggr\vert_{\partial\Omega}=0\biggr\},
\end{equation}
where $u|_\Gamma$
is the trace of the function $u,$ ${\mathbbm 1}_\Gamma$ is the unity function on $\Gamma,$ and $n_+$ is the exterior normal to $\partial\Omega.$
On $\dom(A)$ we set the action of the operator $A$ by the formula
\begin{equation*}
	A\binom{u_+}{\eta}=
	\left(\begin{array}{c}-\Delta u_+\\[0.6em]
		\dfrac{1}{\sqrt{|\Omega_-|}}\mathop{\mathlarger{\mathlarger{\int}}}_{\!\!\!\Gamma}\dfrac{\partial u_+}{\partial n_+}
	\end{array}\right).
\end{equation*}

In the context of the paper \cite{GrandePreuve}, which concerns periodic graphs with high contrast, an analogue of the operator $A$ emerges. We focus on that for the remainder of this paper. This choice allows us to carry out all the necessary computations explicitly, thus facilitating an added transparency of the exposition. (We expect the key outcomes of our study to be transferable to the PDE setup, as the structure of the operators involved remains unchanged -- the related analysis will be the subject of a future publication.)    

 For differentiation $\partial$ and $\tau\in[-\pi,\pi),$ consider the 
 operator ${\partial_\tau}:=\partial+{\rm i}\tau.$  Problems of multiscale analysis of the behaviour of heterogeneous media with high contrast lead to differential operators on an interval $(0,l)$ of the form
\begin{equation}
A\binom{u}{\beta}=\left(\begin{matrix}-\partial_\tau^2 u\\[0.2em]
	-\eta^{-1}Du+\gamma\eta^{-2}\beta\end{matrix}\right),
\label{Aop}
\end{equation}
where  $\eta\in{\mathbb R}\setminus\{0\},$ $\gamma>0$, and
\begin{equation}
Du:={\partial_\tau}u(0)-\omega{\partial_\tau}u(l), \qquad\omega\in{\mathbb C},\ |\omega|=1.
\label{Du_form}
\end{equation}
The domain of the operator $A$ in $L^2(0, l)\oplus{\mathbb C}$ is defined as follows:
\begin{equation}
{\rm dom}(A)=\left\{\binom{u}{\beta}\in W^{2,2}(0,l)\oplus{\mathbb C}:\  u(0)=\omega u(l)=\eta^{-1}\beta\right\}.
\label{domain}
\end{equation}
The pair $(u,\beta)^\top$ describes the approximation of the solution to a second-order differential equation with contrasting parameters in a ``resonant" asymptotic regime, see our recent papers \cite{Physics, GrandePreuve} as well as \cite{ChEK_future} for a similar object in the PDE context. The components $u$ and $\beta$ correspond to the leading-order behaviour on the ``soft" (resonant) and ``stiff" parts of the composite medium, capturing the fact that the soft part supports vibrations of relatively small wavelengths in relation to the stiff part. We next describe the context in which (\ref{Aop}) emerges in more detail.


\subsection{The operator $A$ as the dilation of a generalised resolvent}
\label{gen_res_sec}

The operator (\ref{Aop})--(\ref{domain}) is the \v{S}traus-Neumark dilation for the solution operator ${\mathcal R}(z)$ of the problem
\begin{equation}
	\begin{aligned}
		&-\partial_\tau^2u-zu=f,\\[0.3em]
		&u(0)=\omega u(l),\qquad Du=(\gamma-\eta^2z)u(0),
	\end{aligned}
	\label{gen_res}
\end{equation}
where the relationship between $u(0)$ (and hence $u(l)$) to $\beta$ given in (3) has been used. 
Its action is the composition of the solution to  
\begin{equation*}
A\binom{u}{\beta}-z\binom{u}{\beta}=\binom{f}{0}
\end{equation*}
and the orthogonal projection $P_{\widetilde{H}}$ onto $\widetilde{H}:=L^2(0,l)\oplus\{0\}.$
 On the abstract level, this is expressed as follows:
\begin{equation}
{\mathcal R}(z)=P_{\widetilde{H}}(A-zI)^{-1}\bigr\vert_{\widetilde{H}},
\label{gen_res_rep}
\end{equation}
where $\widetilde{H}$ is identified with $L^2(0,l),$ and therefore in the terminology introduced by \cite{Naimark1940, Strauss}, the operator ${\mathcal R}(z)$ is a generalised resolvent.


	Note that in the BVP (\ref{gen_res}) the spectral parameter is present not only in the differential equation but also in the boundary conditions. In fact, (\ref{gen_res}) can be written in the form\footnote{Indeed, one can set, e.g. (see \cite[Appendix B]{GrandePreuve}), 
\[
\widetilde{\Gamma}_1u=\frac{1}{\sqrt{2}}\left(\begin{array}{c}\partial_\tau u(0)-\omega\partial u(l)\\[0.4em]
	-u(0)+\omega u(l)\end{array}\right),\qquad
\widetilde{\Gamma}_0u=\frac{1}{\sqrt{2}}\left(\begin{array}{c}u(0)+\omega u(l)\\[0.4em]
	\partial_\tau u(0)+\omega\partial u(l)
	\end{array}\right).
\]	
Then the equation (\ref{gen_prob2}) with 
\[
B(z)=\left(\begin{array}{cc}(\gamma-\eta^2z)/2&0\\[0.4em]0&0\end{array}\right)
\]
is shown to be equivalent to (\ref{gen_res}).}
	\begin{equation}
		\widetilde{A}_{\text{\rm max}}u-z u =f,\qquad 
			\widetilde{\Gamma}_1 u = B(z)\widetilde{\Gamma}_0 u,\quad\qquad u\in W^{2,2}(0,l),
		\label{gen_prob2}
	\end{equation}
where  $\widetilde{A}_{\rm max}$ is the operator generated by the differential expression $\partial_\tau^2$ on the domain $W^{2,2}(0, l),$ appropriately chosen operators $\widetilde{\Gamma}_0,$ $\widetilde{\Gamma}_1:W^{2,2}(0, l)\to{\mathbb C}^2$ satisfy Green's identity for all $u,v\in W^{2,2}(0,l):$
\begin{equation*}
	\begin{aligned}
\int_{0}^l\bigl(-\partial_\tau^2u\overline{v}+u\overline{\partial_\tau^2v}\bigr)\equiv\bigl\langle\widetilde{A}_{\rm max}u, v\bigr\rangle_{L^2(0, l)}-\bigl\langle u,\widetilde{A}_{\rm max}v\bigr\rangle_{L^2(0,l)}
=\bigl\langle\widetilde{\Gamma}_1u,\widetilde{\Gamma}_0v\bigr\rangle_{{\mathbb C}^2}-\bigl\langle\widetilde{\Gamma}_0u,\widetilde{\Gamma}_1v\bigr\rangle_{{\mathbb C}^2},
\end{aligned}
\end{equation*}
and $-B(z)$ is an operator-valued $R$-function, i.e., $B(z)$ is analytic in $\mathbb C_+\cup \mathbb C_-$ with $\Im z \Im B(z)\leq 0.$  
	The abstract result of  
 \cite{Strauss} ensures that the solution to any BVP with this property is a generalised resolvent, i.e., it admits a representation of the form (\ref{gen_prob2}).
 Thus the link between (\ref{Aop})--(\ref{domain}) and (\ref{gen_res_rep}) (hence (\ref{gen_res})) is a particular example of a general result of Neumark and \v{S}traus. On the other hand, problems of both types (\ref{gen_prob2}) and (\ref{Aop})--(\ref{domain}) emerge in the process of deriving operator-norm asymptotic approximations for problems of high contrast (``resonant'') homogenisation \cite{Physics, GrandePreuve}. In particular, the problem (\ref{gen_res}) emerges from the asymptotic analysis of the generalised resolvent obtained by projecting the original operator onto the soft component, whereas the problem (\ref{Aop})--(\ref{domain}) turns out to be (up to a unitary equivalence) the asymptotic limit of the family of the complete operator resolvents. While on the abstract level it is not possible to show that the convergence of the generalised resolvents implies the convergence of their Neumark-\v{S}traus dilations, this happens to be the case in all homogenisation setups studied to date. 
 
 	Over the recent years there have been several attempts to provide an explicit construction of the Neumark-\v{S}traus dilation for several classes of generalised resolvents; among the relevant works we would like to point out \cite{Shkalikov_1983, OQS_Malamud, Trace_formulae_Malamud, Figotin_Schenker_2005, Figotin_Schenker_2007b}. This activity has been motivated by the growing interest to the mathematical analysis of highly dispersive media. However, all these constructions stop short of obtaining the functional model representation for the said dilation.
 	
 	On the other hand, in many physically relevant contexts, including that of homogenisation, families of generalised resolvents emerge in a natural way for which the asymptotic expansion with respect to the (small) length-scale parameter yields a leading-order term that can be represented by a generalised resolvent with a linear dependence on the spectral parameter $z.$ From the physics perspective, this corresponds to an effective model of the medium that includes zero-range potentials with an internal structure \cite{AdamyanPavlov, Pavlov_internal_structure}. It can be argued that the linearity of the  impedance in $z$  is essentially equivalent to the model where these zero-range potentials represent point dipoles \cite{CherKis}. If one takes into account higher-order terms in the mentioned asymptotic expansion, one is able to pass from dipole models of effective media to more general multipole ones. While in the present work we focus on the dipole case, the development of the general multipole theory is extremely topical from the point of view of describing metamaterials and can be treated on the basis of the mathematical approach presented here, with a natural replacement of the scalar model by a matrix one.

 	In summary, the ``dipole" homogenisation regime  offers a simple, yet physically relevant in certain frequency regimes, model for which the construction of the dilation can be carried out explicitly, by essentially adding a one-dimensional subspace. 

    This suggests, in particular, that the formulation (\ref{gen_res}) is of a generic type, applicable to a variety of physical contexts, including the Maxwell system of electromagnetism and linearised elasticity. We anticipate that in all those setups it will yield new interesting physical and mathematical effects, which, in our opinion, justifies our interest to such a simple-looking BVP as (\ref{gen_res}).   
  
   We next consider a periodic metric graph that, upon the application of a suitable unitary mapping (``Gelfand transform"), yields an operator of the form (\ref{gen_res}). We then introduce a boundary triple that leads to the so-called $M$-function, which is the key ingredient of the functional model constructed Sections \ref{sec:functional-model},\,\ref{expl_sec}.    

\subsection{Infinite-graph setup and Gelfand transform}
\label{sec_2}
Consider a graph ${\mathbb G}_\infty,$ periodic in one direction, so that ${\mathbb G}_\infty+\ell={\mathbb G}_\infty,$ where $\ell$ is a fixed vector defining the graph axis.
Let the periodicity cell  ${\mathbb G}_\varepsilon$ be a finite compact graph of total length $\varepsilon\in(0,1),$ and denote by
$e_j,$ $j=1,2,\dots n,$ $n\in{\mathbb N},$ its edges. For each $j=1,2,\dots, n,$ we identify $e_j$ with the interval $[0,\varepsilon l_j],$ where $\varepsilon l_j$ is the length of $e_j.$  We associate with the graph ${\mathbb G}_\infty$ the Hilbert space
$$
L_2({\mathbb G}_\infty):=\bigoplus\limits_{{\mathbb Z}}\bigoplus\limits_{j=1}^n L_2(0, \varepsilon l_j).
$$

Consider also a family $\{A^\varepsilon\}_{\varepsilon>0}$ of operators 
in $L_2({\mathbb G}_\infty),$ generated by second-order differential expressions
$-a^\varepsilon\partial^2,$
with positive ${\mathbb G}_\varepsilon$-periodic coefficients  $a^\varepsilon$ on ${\mathbb G}_\infty,$
and defined on the domain ${\rm dom}(A^\varepsilon)$ describing the ``natural" coupling conditions at the vertices of
${\mathbb G}_\infty:$
\begin{equation}
	{\rm dom}(A^\varepsilon)=
	\Bigl\{
	u\in\bigoplus\limits_{e\in{\mathbb G}_\infty}W^{2,2}\bigl(e):\ u
	\text{\ continuous,}\ \sum_{e\ni V}\sigma_ea^\varepsilon u'(V)=0\ \ 
	\forall V\in{\mathbb G}_\infty\Bigr\}.
	\label{Atau}
\end{equation}
In (\ref{Atau}) the summation is carried out over the edges $e$
sharing the vertex $V,$ the coefficient $a^\varepsilon$ in the vertex condition is calculated on the edge $e,$ and $\sigma_{ e}=-1$ or $\sigma_{e}=1$ for $e$  incoming or outgoing  for $V,$ respectively. 
The matching conditions \eqref{Atau} represent the combined conditions of continuity of the function and of vanishing sums of its co-normal derivatives at all vertices (i.e., the so-called Kirchhoff conditions).

Applying to the operators $A^\varepsilon$ a suitable version of the Gelfand transform \cite{Gelfand, KisRyad}, one obtains a two-parametric family of operators $A^\varepsilon_\tau,$ $\tau\in[-\pi,\pi) ,$ $\varepsilon>0,$ defined on the space of $L^2$-functions on a ``unit cell" $\mathbb G$ of size one, obtained from the ``$\varepsilon$-cell" ${\mathbb G}_\varepsilon$ by a simple scaling ${\mathbb G}_\varepsilon\ni x\mapsto y=x/\varepsilon\in{\mathbb G}.$
 More precisely, at each vertex $V$ of ${\mathbb G}$ there exists a list of unimodular ``weights" $\{w_V(e)\}_{e\ni V},$ cf. \cite{Physics}, defined  as a finite collection of values corresponding to the edges adjacent to $V$.
For each $\tau\in[-\pi,\pi)$, the fibre operator $A^\varepsilon_\tau$  is generated by the differential expression $a^\varepsilon\varepsilon^{-2}\partial_\tau^2$
on the domain
\begin{equation*}
	\begin{aligned}
		{\rm dom}(A^\varepsilon_\tau)=
		\Big\{
		v\in\bigoplus\limits_{e\in {\mathbb G}}W^{2,2}\bigl(e):&\ 
		w_V(e)v|_e(V)=w_V(e')v|_{e'}(V)
		\text{\ for all } e,e' \text{ adjacent to } V,\\[0.0em] \ &\sum_{e\ni V}\widetilde{\partial}_\tau v(V)=0\ \ \
		{\rm for\ each\ vertex}\ V\Big\},
	\end{aligned}
\end{equation*}
where $\widetilde{\partial}_\tau v(V)$ stands for the ``weighted co-derivative'' $\sigma_{e}w_V(e)a^\varepsilon\varepsilon^{-2}\partial_\tau v$   
of the function $v$ on the edge $e,$ calculated at the vertex $V.$



\subsection{An example of operator on a graph and it norm-resovent approximation}

\label{section:examples}

The periodic graph considered, its periodicity cell and the result of Gelfand transform is shown in Fig.\,\ref{fig:1}. Denote by $a_j,$ $j=1,2,3,$ the values of $a^\varepsilon$ on the edges $e_j,$ $j=1,2,3,$ and assume for simplicity that $a_j=1.$
\begin{figure}[h!]
	\begin{center}
		\includegraphics[scale=0.7]{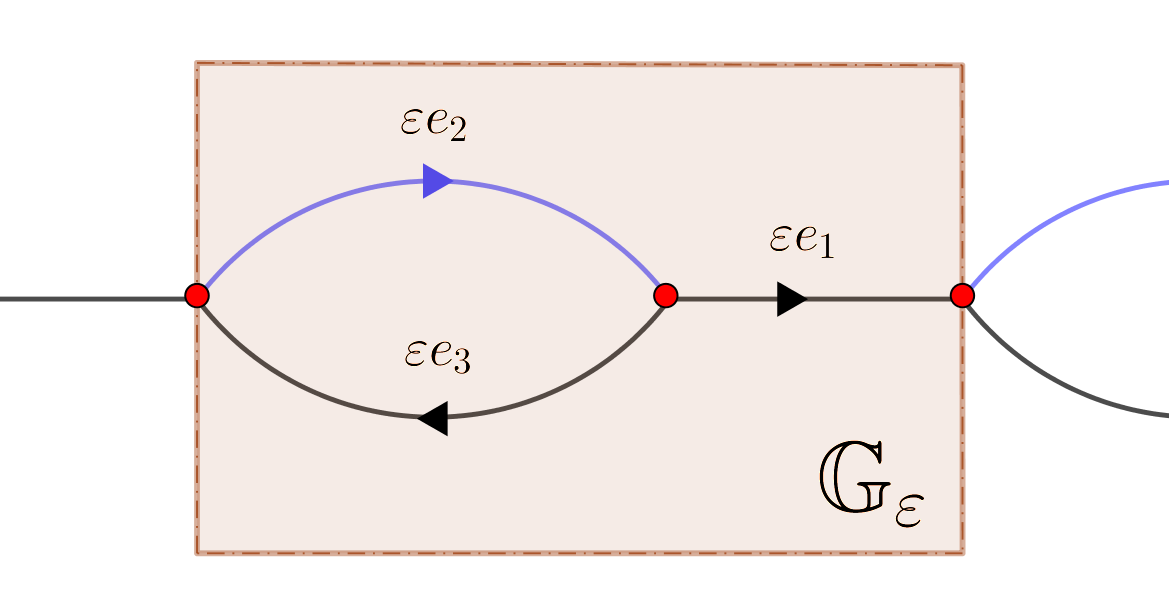}
		\includegraphics[scale=0.7]{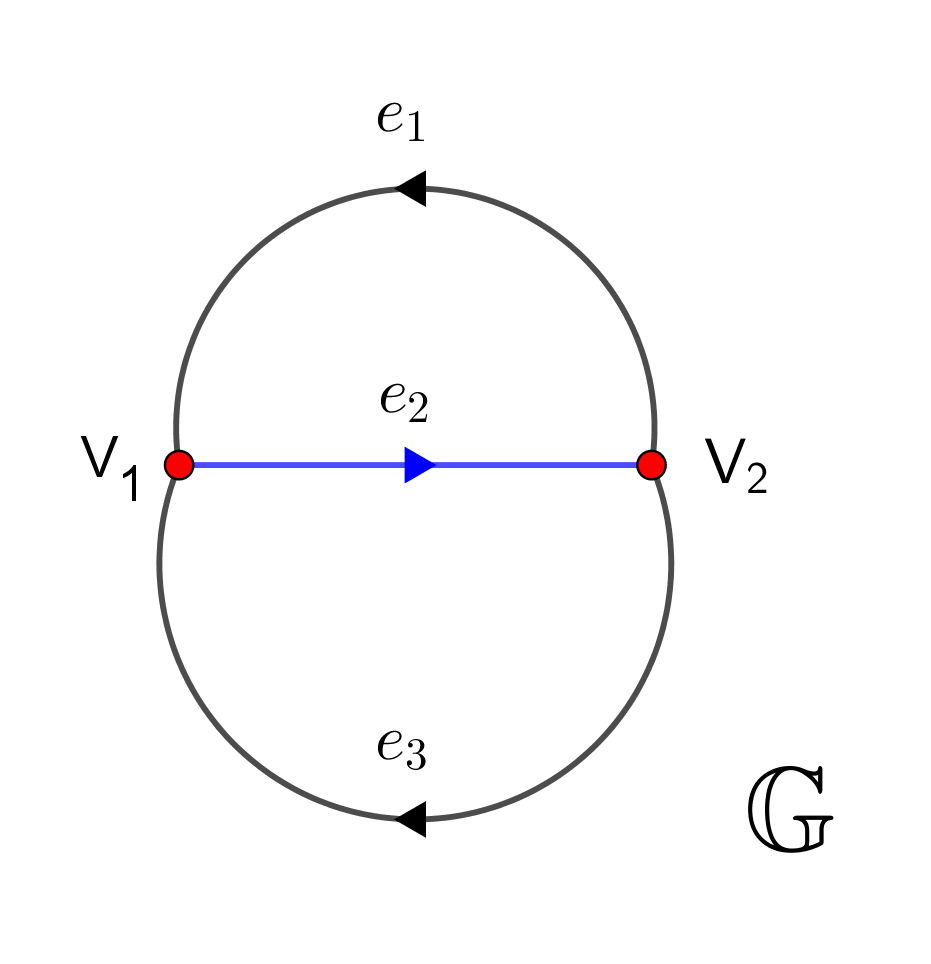}
	\end{center}
	\caption{{\scshape Example of a periodic graph with contrast.} {\small The infinite  graph $\mathbb G_\infty$ and the ``period" $\mathbb G_\e$ are outlined on the left; the graph unit cell $\mathbb G$ obtained after applying the Gelfand transform is shown on the right. The soft component is drawn in blue.}\label{fig:1}}
\end{figure}
The unimodular values $w_{V_k}(e_j),$ $j=1,2,3,$ $k=1,2,$ are then chosen as follows:
\begin{equation*}
	\begin{gathered}
		\{w_{V_1}(e_j)\}_{j=1}^3=\bigl\{1,1,{\rm e}^{{\rm i}\tau(l_2+l_3)}\bigr\},\quad \{w_{V_2}(e_j)\}_{j=1}^3=\bigl\{{\rm e}^{{\rm i}\tau l_3},1,1\bigr\}
	\end{gathered}
\end{equation*}

For all $\tau\in[-\pi, \pi),$ consider an operator $\mathcal A_{\rm hom}^\tau$ on $L^2(0,l_2)\oplus\mathbb C,$
defined as follows. Denote 
\[
\xi_\tau=-\frac{a_1}{l_1}{\rm e}^{{\rm i}\tau(l_1+l_3)}-\frac{a_3}{l_3}{\rm e}^{-{\rm i}\tau l_2}.
\]
The domain $\dom(\mathcal A_{\hom}^\tau)$ is set to be
\begin{equation*}
	\dom\bigl(\mathcal A_{\hom}^\tau\bigr)=\biggl\{(u,\beta)^\top\in L^2(0,l_2)\oplus\mathbb C:\  u\in W^{2,2}(0,l_2),\  u(0)=-\frac {\overline{\xi}_\tau}{|\xi_\tau|}u(l_2)=\frac{\beta}{\sqrt{l_1+l_3}}\biggr\}.
\end{equation*}
On $\dom(\mathcal A_{\hom}^\tau)$ the action of the operator is set by
$$
\mathcal A_{\hom}^\tau\binom{u}{\beta}=
\left(\begin{array}{c}-\partial_\tau^2u
	\\[0.8em]
	-\dfrac{1}{\sqrt{l_1+l_3}}
	\biggl({\partial_\tau} u(0) + \dfrac {\overline{\xi}_\tau}{|\xi_\tau|}\partial_\tau u(l_2)\biggr)+\bigl(l_1+l_3\bigr)^{-1}\biggl(\dfrac {l_1}{a_1}+\dfrac {l_3}{a_3}\biggr)^{-1}\biggl(\dfrac{\tau}{\e}\biggr)^2 \beta
\end{array}\right).
$$

\begin{theorem}[\cite{GrandePreuve}]
	\label{thm:ex1}
	Denote $H:=\oplus_{j=1}^3 L^2(0, l_j).$ There exists $C>0,$ independent of $\varepsilon$ and $\tau,$ such that 
	$$
	\bigl\Vert(A_\e^\tau-z)^{-1}-\Psi^* (\mathcal A_{\hom}^\tau-z)^{-1}\Psi\bigr\Vert_{H\to H}\le C\varepsilon^2,
	$$
	where $\Psi$ is a partial isometry  from $H$ to $L^2(0,l_2)\oplus\mathbb C.$
\end{theorem}


Clearly the operator $A_{\hom}^\tau$ is of the form (\ref{Aop})--(\ref{domain}) with 
\begin{equation}
l=l_2,\quad\eta=\sqrt{l_1+l_3},\quad \gamma=(l_1a_1^{-1}+l_3a_3^{-1})^{-1}(\tau/\varepsilon)^2,\quad\omega=-\overline{\xi}_\tau/|\xi_\tau|.
\label{parameters}
\end{equation}

\subsection{Abstract boundary triples}
\label{BT_theory}




Our approach is
based
on the theory of boundary triples \cite{Gor,Ko1,Koch,DM}, applied
to the class of operators introduced above. We next recall two fundamental
concepts of this theory, namely the boundary triple and the generalised Weyl-Titchmarsh
matrix function.

\begin{definition}
	\label{triple_def}
	Suppose that 
	${A}_{\rm max}$ is the adjoint to a densely defined symmetric operator ${A}_{\rm min}$ on a separable Hilbert space ${H}$ (``physical region space") and  
	that ${\Gamma}_0,$ ${\Gamma}_1$ are linear mappings of ${\rm dom}({A}_{\max})\subset{H}$
	to a separable Hilbert space ${\mathcal{H}}$ (``boundary space").
	
	A. The triple
	$({\mathcal{H}}, {\Gamma}_0,{\Gamma}_1)$ is called a boundary
		triple for the operator ${A}_{\max}$  
	if:
	\begin{enumerate}
		\item For all $u,v\in {\rm dom}({A}_{\max})$ one has the second Green's identity
		\begin{equation*}
			\langle {A}_{\max} u,v \rangle_{{H}} -\langle u, {A}_{\max} v \rangle_{{H}} = \langle{\Gamma}_1 u, {\Gamma}_0
			v\rangle_{{\mathcal{H}}}-\langle{\Gamma}_0 u, {\Gamma}_1 v\rangle_{{\mathcal{H}}}.
		\end{equation*}
		\item The mapping
		${\rm dom}({A}_{\max})\ni u\longmapsto ({\Gamma}_0 u,
		{\Gamma}_1 u)\in{{\mathcal H}}\oplus{{\mathcal H}}$
		is onto.
	\end{enumerate}
	

	B. The operator-valued Herglotz\footnote{For a definition and properties of Herglotz functions, see, e.g., \cite{MR0328627KK, Nussenzveig, GT, GKMT, ABT}.} function ${\mathfrak m}={\mathfrak m}(z),$ defined by
	\begin{equation}
		\label{Eq_Func_Weyl}
		{\mathfrak m}(z){\Gamma}_0 u_{z}={\Gamma}_1 u_{z}, \ \
		u_{z}\in \ker ({A}_{\max}-z),\  \ z\in
		\mathbb{C}_+\cup{\mathbb C}_-,
	\end{equation}
	is referred to as the $M$-function of the operator
	${A}_{\max}$ with respect to the triple $({\mathcal{H}}, {\Gamma}_0,{\Gamma}_1)$.

	C. A non-trivial extension ${A}_B$ of the operator ${A}_{\min}$ such
	that ${A}_{\min}\subset {A}_B\subset{A}_{\max}$  is called almost solvable if there exists a boundary triple
	$({\mathcal{H}}, {\Gamma}_0, {\Gamma}_1)$ for ${A}_{\max}$ and a bounded
	linear operator $B$ defined on ${\mathcal{H}}$ such that for every
	$u\in {\rm dom}({A}_{\max})$ one has $u\in {\rm dom}({{A}_B})$ if and only if ${\Gamma}_1 u=B{\Gamma}_0 u.$
	
\end{definition}

In what follows, we use the boundary triple approach to the extension theory of
symmetric operators with equal deficiency indices (see
\cite{Derkach} for a review of the subject), which is particularly useful
in the study of extensions of ordinary differential operators of second order.

\subsection{The boundary triple for the prototype dilation operator}
\label{proto_triple}


Here we aim at constructing a convenient boundary triple for the operator (\ref{Aop})--(\ref{domain}) in the space $H:=L^2(0,l)\oplus{\mathbb C}.$ To this end, consider the following domains for the minimal and maximal (i.e., the adjoint to the minimal) operators corresponding to the same expression (\ref{Aop}):
\begin{align}
{\rm dom}\left(A_{\rm min}\right)&=\left\{\binom{u}{\beta}\in W^{2,2}(0,l)\oplus{\mathbb C}: u(0)=\omega u(l)=\eta^{-1}\beta,\  Du
=0\right\},\nonumber  
\\[0.5em] 
{\rm dom}\left(A_{\rm max}\right)&=\left\{\binom{u}{\beta}\in W^{2,2}(0,l)\oplus{\mathbb C}: u(0)=\omega u(l)\right\},\label{Amax_dom}
\end{align}
where $Du:={\partial_\tau}u(0)-\omega{\partial_\tau}u(l).$

\begin{theorem} The triple $({\mathcal H}, \Gamma_0, \Gamma_1),$ where
\begin{equation}
\mathcal{H}:=\mathbb{C},\qquad\Gamma_0 \binom{u}{\beta} := Du,
\qquad 
\Gamma_1 \binom{u}{\beta} := \frac{\beta}{\eta}-u(0),\quad\qquad \binom{u}{\beta}\in W^{2,2}(0,l)\oplus{\mathbb C},
\label{triple_expl}
\end{equation}
is a boundary triple for the operator $A_{\rm max}$ defined by the expression (\ref{Aop}) on the domain (\ref{Amax_dom}). 
\end{theorem}

\begin{proof}The second property of the triple in Definition \ref{triple_def} is verified immediately, and the following calculations show that the second Green's identity holds as well:
\begin{align*}
	&\left\langle A_{\rm max}\binom{u}{\beta},\binom{v}{\zeta}\right\rangle_H-\left\langle \binom{u}{\beta},A_{\rm max}\binom{v}{\zeta}\right\rangle_H=
	\\
	&=-\int\limits_0^l \partial_\tau^2 u \bar{v}dx-\left(\frac 1{\eta}Du+\frac{\gamma}{\eta^2}\beta\right)\bar\zeta+\int\limits_0^l u \overline{\partial_\tau^2 v} dx+\beta \overline{\left(\frac 1{\eta}Dv+\frac{\gamma}{\eta^2}\zeta\right)}
	\\
    &=-\int\limits_0^l u \overline{\partial_\tau^2 v} dx +\left[u \overline{{\partial_\tau}v}-{\partial_\tau}u \overline{v}\right]_{0}^{l}-
    \left(\frac{1}{\eta}Du
    +\frac{\gamma}{\eta^2}\beta\right)\overline{\zeta}
	+\int\limits_0^l u \overline{\partial_\tau^2 v} dx+\beta 
	\overline{\left(\frac{1}{\eta}Dv+\frac{\gamma}{\eta^2}\zeta\right)}
	\\
	&=\left(u(l)\overline{{\partial_\tau}v(l)}-u(0) \overline{{\partial_\tau}v(0)}\right)-\left({\partial_\tau}u(l) \overline{v(l)}-
	{\partial_\tau}u(0) \overline{v(0)}\right)+\frac {\beta}{\eta} \overline{Dv}-\frac {\bar{\zeta}}{\eta}Du
	\\[0.6em]
	&= \Gamma_1\binom{u}{\beta}\overline{\Gamma_0\binom{v}{\zeta}}-\Gamma_0 \binom{u}{\beta}\overline{\Gamma_1\binom{v}{\zeta}}.
\end{align*}
\end{proof}
Let us next calculate the corresponding $M$-function, which is defined by the property (cf. (\ref{Eq_Func_Weyl}))  
$$
{\mathfrak m}(z)\Gamma_0\binom{u_z}{\beta_z}=\Gamma_1\binom{u_z}{\beta_z},\qquad\quad \binom{u_z}{\beta_z}\in \ker(A_{\rm max}-zI).
$$

\begin{theorem}
The $M$-function of the operator
$A_{\max}$ with respect to the triple (\ref{triple_expl})
is given by
\begin{equation}
	{\mathfrak m}(z)=-\frac{\sin\sqrt{z}l}{2\sqrt{z}\bigl(\Re({\rm e}^{{\rm i}\tau l}\bar{\omega})-\cos\sqrt{z}l\bigr)}-\frac{1}{\eta^2 z-\gamma}.
	\label{Mfunction_expl}
\end{equation}

\end{theorem}

\begin{proof}
The general solution of the spectral problem
$$
\left\{
\begin{array}{ll}
	-\partial_\tau^2 u_z=z u_z,\\[0.4em]
	-\eta^{-1}Du+\gamma\eta^{-2}\beta_z= z\beta_z,
\end{array}
\right.$$
is given by
$$
	u_z = {\rm e}^{-{\rm i}\tau x}\bigl(C_1{\rm e}^{{\rm i}\sqrt{z}x}+C_2{\rm e}^{-{\rm i}\sqrt{z}x}\bigr),
	\qquad
	\beta_z=\dfrac{\eta}{\eta^2z-\gamma}Du,
$$
where the branch of the square root is chosen so that $\sqrt{z}$ is real for real positive $z$. 

Normalising $u_z$ by the condition 
\begin{equation}
u_z(0)=\omega u_z(l)=1,
\label{normalisation}
\end{equation}
 we obtain  $\Gamma_0 u_z= Du_z$, $\Gamma_1 u_z=\eta^{-1}\beta_z-1,$
and hence 
\begin{equation}
{\mathfrak m}(z)=\frac{\Gamma_1 u_z}{\Gamma_0 u_z}=-\frac{1}{\eta^2 z-\gamma}-\frac{1}{Du_z}.
\label{M_first}
\end{equation}

It remains to determine the values $C_1,$ $C_2$ for $u_z$ satisfying (\ref{normalisation}) and hence evaluate $Du_z.$ To this end, we write 
$$
	u_z(0)=C_1+C_2=1,\qquad\quad
	u_z(l)={\rm e}^{-{\rm i}\tau l}\bigl(C_1{\rm e}^{{\rm i}\sqrt{z}l}+C_2{\rm e}^{-{\rm i}\sqrt{z}l}\bigr)=\bar{\omega},
$$
whence 
$$
	C_1=\dfrac{\bar{\omega}{\rm e}^{{\rm i}\tau l}-{\rm e}^{-{\rm i}\sqrt{z} l}}{2{\rm i}\sin\sqrt{z}l},\qquad\quad 
	C_2=\frac{{\rm e}^{{\rm i}\sqrt{z}l}-\bar{\omega}{\rm e}^{{\rm i}\tau l}}{2{\rm i}\sin(\sqrt{z}l)}.
$$
It follows that
$$
u_z(x)=\frac{{\rm e}^{-{\rm i}\tau x}}{\sin\sqrt{z}l} \left(\bar{\omega}e^{{\rm i}\tau l}\sin\sqrt{z}x +\sin\sqrt{z}(l-x)\right),\qquad x\in[0,l],
$$
and, in particular,
$$
Du_z=\frac{\sqrt{z}}{\sin\sqrt{z}l} \left(\bar{\omega}{\rm e}^{{\rm i}\tau l}-\cos\sqrt{z}l -\omega(\bar{\omega}\cos\sqrt{z}l -{\rm e}^{-{\rm i}\tau l}) \right)=
\frac{2\sqrt{z}\bigl(\Re({\rm e}^{{\rm i}\tau l}\bar{\omega})-\cos\sqrt{z}l\bigr)}{\sin\sqrt{z}l}.
$$
Combining this with (\ref{M_first}) finally yields (\ref{Mfunction_expl}).
\end{proof}


\section{Spectral form of the functional model for the \v{S}traus-Neumark dilation}
\label{sec:functional-model}

	The first (and the only known to us) attempt at a construction of the functional model for a generalised resolvent is contained in \cite{PavlovFaddeev}, where a ``5-component" self-adjoint dilation was developed for a (actually, more challenging) problem with an impedance linear in $\sqrt{z}$ rather than $z,$ using methods resembling those employed in the dilation theory for dissipative operators. However, that work stops short of constructing any sort of spectral representation for the named dilation.      
	
	Setting out to construct a spectral form for the dilation in our case, we draw our inspiration in essentially the same pool of ideas but, instead of constructing a 5-component model like in \cite{PavlovFaddeev}, we achieve our goal in two steps. First, facilitated by the linear in $z$ form of the impedance, we construct an out-of-space self-adjoint extension of the associated symmetric operator (i.e., the one obtained by ``restricting" the generalised resolvent), so that the named extension is the Neumark-\v{S}traus dilation of our generalised resolvent. Second, considering a fixed dissipative extension of the same symmetric operator, we develop its self-adjoint dilation, thereby dilating the underlying space even further. Following this, we utilise an explicit formula describing the resolvent of the Neumark-\v{S}traus dilation constructed at the first step in this ``twice-dilated'' space. The overall success of the strategy is rooted in the fact that the self-adjoint dilation of the dissipative operator introduced at the second step admits an explicit spectral representation. It is in this spectral representation that the action of the self-adjoint Neumark-\v{S}traus dilation takes the simplest form, which can be shown to be a triangular perturbation of a Toeplitz operator \cite{KisNab, CherKisSilva2}. The latter is then used to pass over to a yet another representation, where the original Hilbert space is unitarily equivalent to a space of the class $K_\theta,$ which has been studied in, e.g., \cite{Clark, Aleksandrov_pseudocontinuable, Poltora, Poltora_Sarason, Nikolski}.
	
	Finally, we make use of the fact that the space $K_\theta,$ in its turn, is unitarily equivalent to the $L^2$-space with respect to a Clark measure. We note that alternative constructions to \cite{PavlovFaddeev} have appeared in the literature
	 \cite{Shkalikov_1983, OQS_Malamud, Trace_formulae_Malamud, Figotin_Schenker_2005, Figotin_Schenker_2007b}, which, however, touches neither upon the spectral form of the Neumark-\v{S}traus dilation nor upon the functional model for the associated generalised resolvent.





The first step of the above programme has been carried out in Section \ref{proto_triple}, where the corresponding extension of the minimal symmetric operator has been constructed, the corresponding boundary triple framework has been developed, and the corresponding $M$-function has been computed.

 In order to pursue the second step, we now need to pick a convenient dissipative operator belonging to the class considered, which is the class of all extensions $A_\varkappa,$ 
$\varkappa\in{\mathbb C},$ of $A_{\rm min}$
whose domains are given on the basis of
the boundary triple
$({\mathbb C},\Gamma_1,\Gamma_0)$ for $A_{\rm max}$ as
follows:
\begin{equation}
	\label{eq:extension-by-operator}
	\dom(A_\varkappa):=\bigl\{f\in\dom(A_{\max}):\Gamma_1f=\varkappa\Gamma_0f\bigr\}.
\end{equation} 
 It follows from \cite[Thm.\,2]{Ko1} and
\cite[Chap.\,3 Sec.\,1.4]{Gor} (see also an alternative
formulation in \cite[Thm.\,1.1]{MR2330831}, and
\cite[Sec.~14]{MR2953553}) that $A_\varkappa$ is maximal,  i.\,e., $\rho(A_\varkappa)\ne\emptyset$. For the construction of the Pavlov model, we need to consider one selected dissipative operator, given by (\ref{eq:extension-by-operator}) with $\varkappa={\rm i}.$

	It was shown by Ryzhov \cite{MR2330831} that the characteristic function ${\mathfrak s}$ of \v{S}traus for the operator $A_{\I}$ is
	given by	
\begin{equation}
	{\mathfrak s}(z)=1-\frac{2{\rm i}}{{\rm i}+{\mathfrak m}(z)},\qquad z\in\complex_+.
	\label{SviaM}
\end{equation}	
Thus, the characteristic function is the Cayley transform of the $M$-function ${\mathfrak m},$ 
cf. \cite{Pavlov_cont_spec_res}.
	Based on the material presented in Section 2.5 or by a standard argument, one verifies that $\mathfrak{s}$ is analytic in $\complex_+$ and, for each
	$z\in\complex_+$, $\abs{\mathfrak{s}}\le 1$. Therefore, by invoking the classical Fatou theorem, see {\it e.g.} \cite{MR2760647},  the function $\mathfrak{s}$ has a nontangential limit almost everywhere on the real line, 
	which we will henceforth denote by $\mathfrak{s}(k),$ $k\in{\mathbb R}.$ However, in our case its analytic properties in the vicinity of the real line are in fact much better, which we discuss and take advantage of below. 



The next definitions apply to arbitrary values of $\varkappa,$ although in our analysis we will require the objects pertaining to $\varkappa=0$ and $\varkappa={\rm i}.$
We abbreviate
\begin{equation*}
	\theta_\varkappa(z):=1-2({\rm i}-\mathfrak m(z))^{-1}
	\chi_\varkappa^+\,,\qquad z\in\complex_-,\qquad\qquad
	\widehat{\theta}_\varkappa(z):=1-2({\rm i}+\mathfrak m(z))^{-1}\chi_\varkappa^-\,,\qquad z\in\complex_+,
\end{equation*}
where
\begin{equation}
	\label{eq:chi-def}
	\chi_\varkappa^\pm:=\frac{1\pm\varkappa}{2},
\end{equation}

The definition of the characteristic function ${\mathfrak s}$ and the
fact that $\mathfrak m$ is a Herglotz function \cite{MR0328627KK} allow us to write $\theta_\varkappa(z)$ and
$\widehat{\theta}_\varkappa(z)$ in terms of $\mathfrak s$ as follows:
\begin{equation}
	\theta_\varkappa(z)=1+(\cc{\mathfrak s(\cc{z})}-1)\chi_\varkappa^+,\quad\quad
	\,z\in\complex_-,\qquad\qquad
	\widehat{\theta}_\varkappa(z)=1+(\mathfrak s(z)-1)\chi_\varkappa^-,\quad\quad
	\,
	\,z\in\complex_+.
\label{eq:alternate-for-theta-hat}
\end{equation}

We will next use an explicit construction of the functional model for the operator family $A_\varkappa,$
introduced in \cite{MR0510053, MR0365199, Drogobych} and further
developed in \cite{MR573902, Ryzhov_closed, Ryzh_ac_sing, Tikhonov}. As the objects introduced above, it applies to arbitrary values of $\varkappa,$ although henceforth we only utilise it for the case $\varkappa=0.$  

 Our immediate goal is to represent the self-adjoint dilation \cite{MR2760647} of the dissipative operator $A_{\rm i}$ as an operator of multiplication. To this end, one first constructs a three-component model of the dilation, following Pavlov's procedure \cite{MR0365199, MR0510053, Drogobych} and then explicitly defining a unitary mapping to the so-called ``symmetric'' representation ${\mathfrak H}$ of the dilation. Namely, one starts with the Hilbert space   
	\begin{equation*}
		\mathscr{H}=
		L^2(\reals_-)\oplus H\oplus L^2(\reals_+),
	\end{equation*}
	and the self-adjoint operator $\mathscr{A}$ in  $\mathscr{H}$ such that
	\begin{equation*}
		P_{H}(\mathscr{A}-zI)^{-1}\eval{H}=
		(A_{\I}-z I)^{-1}\,, \qquad z\in\complex_-,
	\end{equation*}
where $\eval{H}$ and $P_{H}$ stand for the restriction to and the orthogonal projection onto the subspace $\{0\}\oplus H\oplus\{0\},$ which we identify with $H.$  
		Then, as in the case of additive non-selfadjoint perturbations
	\cite{MR573902}, it is established
	\cite[Thm.\,2.3]{MR2330831} that
	there exists
	an isometry $\Phi:\mathscr{H}\to{\mathfrak H}$
	such that
	\[
	\Phi(\mathscr{A}-z I)^{-1}=(\cdot-z)^{-1}\Phi.
	\]
Next, we shall recall how this construction is made explicit in our particular case. 






Following the argument of \cite[Thm.\,1]{MR573902}, it is shown in
\cite[Lem.\,2.4]{MR2330831} that
\begin{equation}
	\label{eq:naboko-thm-1}
\Gamma_0(A_{\I}-\cdot I)^{-1}h\in H_-^2\quad\text{and}
	\quad\Gamma_0(A_{\I}^*-\cdot I)^{-1}h\in H_+^2,
\end{equation}
where $H_\pm^2$ are the standard Hardy classes, see {\it e.g.} \cite[Sec.\,4.8]{MR822228}. 
Further, for a
two-component vector function 
$(\widetilde{g},g)^\top$ 
taking values
in $\mathbb{C}^{2}=\mathbb{C}\oplus\mathbb{C}$, one considers the integral
\begin{equation}
	\label{eq:inner-in-functional}
	\int_\reals\inner{\begin{pmatrix} 1 & \cc{\mathfrak{s}(s)}\\
			\mathfrak s(s) &
			1\end{pmatrix}\binom{\widetilde{g}(s)}{g(s)}}{\binom{\widetilde{g}(s)}{g(s)}}_{\mathbb{C}^{2}}ds,
\end{equation}
which is nonnegative, due to the contractive properties of
$\mathfrak s$.
The space
\begin{equation*}
	\mathfrak{H}:=L^2\Biggl(\mathbb{C}^{2}; \begin{pmatrix}
		1 & \cc{\mathfrak s}\\
		\mathfrak s & 1
	\end{pmatrix}\Biggr)
\end{equation*}
is the completion of the linear set of two-component vector functions
$(\widetilde{g},g)^\top: {\mathbb R}\to\mathbb{C}^{2}$ in the norm
(\ref{eq:inner-in-functional}), factored with respect to vectors of zero norm.
Naturally, not every element of the set can be identified with a pair
$(\widetilde{g},g)^\top$ of two independent functions. Still, in what follows we keep the notation $(\widetilde{g},g)^\top$ for the elements of this space.


Another consequence of the contractive properties of the
characteristic function ${\mathfrak s}$ is that for $\widetilde{g},g\in
L^2(\reals)$ one has
\begin{equation*}
	\norm{\binom{\widetilde{g}}{g}}_\mathfrak{H}\ge
	\max\bigl\{\norm{\widetilde{g}+\cc{\mathfrak s}g}_{L^2(\reals)},\\
		\norm{\mathfrak s\widetilde{g}+g}_{L^2(\reals)}\bigr\}.
\end{equation*}
Thus, for every Cauchy sequence
$\{(\widetilde{g}_n, g_n)^\top\}_{n=1}^\infty$ with respect to the $\mathfrak{H}$-topology,
such that $\widetilde{g}_n,g_n\in L^2(\reals)$ for all
$n\in\nats$, the limits of $\widetilde{g}_n+\cc{\mathfrak s}g_n$ and
$\mathfrak s\widetilde{g}_n+g_n$ exist in $L^2(\reals)$, so that
$g_-:=\widetilde{g}+\cc{\mathfrak s}g$ and
$g_+:={\mathfrak s}\widetilde{g}+g$ can always be treated as $L^2(\reals)$ functions.

Consider the following orthogonal subspaces of $\mathfrak{H}:$
\begin{equation*}
	D_-:=
	\begin{pmatrix}
		0\\[0.2em]
		{H}^2_-
	\end{pmatrix}\,,\quad
	D_+:=
	\begin{pmatrix}
		{H}^2_+\\[0.2em]
		0
	\end{pmatrix}\,.
\end{equation*}
We define the space
\begin{equation*}
	K:=\mathfrak{H}\ominus(D_-\oplus D_+),
\end{equation*}
which is characterised as follows (see {\it e.g.} \cite{MR0365199, Drogobych}):
\begin{equation*}
	K=\left\{\begin{pmatrix}
		\widetilde{g}\\
		g
	\end{pmatrix}\in\mathfrak{H}: 
g_-\in  {H}^2_-\,,
	g_+\in
	{H}^2_+\right\}\,.
\end{equation*}
The orthogonal projection $P_K$ onto the subspace
$K$ is given by (see {\it e.g.} \cite{MR0500225})
\begin{equation*}
	P_K
	\begin{pmatrix}
		\widetilde{g}\\
		g
	\end{pmatrix}
	=
	\begin{pmatrix}
		\widetilde{g}-P_+{g_-}\\[0.3em]
		g-P_-g_+
	\end{pmatrix},
\end{equation*}
where $P_\pm$ are the orthogonal Riesz projections in $L^2(\mathbb{R})$ onto
${H}^2_\pm$.

The next theorem is a particular case of \cite[Thm. 4.1]{CherKisSilva1}, which  generalises \cite[Thm.
	2.5]{MR2330831}, and its form is similar to
	\cite[Thm.\,3]{MR573902}, which treats the case of additive
	perturbation (cf.  \cite{Mak_Vas, MR2330831, Ryzh_ac_sing,
		Ryzhov_closed} for the case of possibly non-additive
	perturbations).
	
\begin{theorem}
	\label{thm:rhyzhov2.5}
	Let $R_{\varkappa}(z):=(A_\varkappa-z I)^{-1}$ for $z\in\rho(A_{\varkappa})$.
	\begin{enumerate}[(i)]
		\item  If $z\in\complex_-\cap\rho(A_\varkappa)$ and
		$(\widetilde{g}, g)^\top\in K$, then
		\begin{equation}
			\Phi R_{\varkappa}(z)\Phi^{*}\binom{\widetilde{g}}{g}=P_K\frac{1}{\cdot-z}
			\binom{\widetilde{g}}{g-\chi_\varkappa^+\theta^{-1}_\varkappa(z)
				g_-(z)}\,.
			\label{representation_minus}
		\end{equation}
		\item  If $z\in\complex_+\cap\rho(A_\varkappa)$ and
		$(\widetilde{g}, g)^\top\in K$, then
		\begin{equation}
			\Phi R_{\varkappa}(z)\Phi^{*}\binom{\widetilde{g}}{g}=P_K\frac{1}{\cdot-z}
			\binom{\widetilde{g}-\chi_\varkappa^-\widehat{\theta}^{-1}_\varkappa(z)
				g_+(z)}{g}\,.
			\label{representation_plus}
		\end{equation}
		Here, ${g_\pm}(z)$
		denote the values at $z$ of the analytic continuations of the functions
		$g_\pm\in H_\pm^2$	into the corresponding half-plane. 
	\end{enumerate}
\end{theorem}

 In the work \cite{KisNab}, concerning the matrix model for non-selfadjoint operators with almost Hermitian spectrum, it is shown (see \cite[Theorem 3.3]{KisNab}) that provided ${\mathfrak s}$ is an inner function (which is the case we are dealing with in the present paper), the Hilbert space $H$ is unitarily equivalent to the spaces 
	 \begin{equation*}
	 K_{\mathfrak s}:= H_+^2\ominus {\mathfrak s}H^2_+,\qquad K^\dagger_{\mathfrak s}:= H_-^2\ominus \overline{\mathfrak s}H^2_-.
	 \end{equation*}
	The related unitary mappings are provided by the formulae (cf. (\ref{eq:naboko-thm-1}))
	
	\begin{equation}
	H\ni v\mapsto-\frac{1}{\sqrt{\pi}}\Gamma_0(A_{\rm i}-\cdot)^{-1}v=g_+\in K_{\mathfrak s},\qquad\quad
	H\ni v\mapsto-\frac{1}{\sqrt{\pi}}\Gamma_0(A_{-{\rm i}}-\cdot)^{-1}v=g_-\in K^\dagger_{\mathfrak s}.
\label{unit_form}
	\end{equation}
Also note that the unitary equivalence between $K_{\mathfrak s},$  $K^\dagger_{\mathfrak s}$ can be obtained via the element-wise equality
 \[
 K_{\mathfrak s}={\mathfrak s}K^\dagger_{\mathfrak s},
 \]
 where it is understood that the multiplication by ${\mathfrak s}$ is applied to the traces of $K_{\mathfrak s}, K^\dagger_{\mathfrak s},$ see also the corresponding statement pertaining to operators of BVPs for PDEs in \cite{CherKisSilva2}. 

	

\section{Explicit functional model representation}
\label{expl_sec}

	
	This section contains the main results of the paper, namely the construction of an explicit functional model for the operators $A_\varkappa,$ i.e., a representation of the Hilbert space $H$ as a space of square summable functions over a measure with respect to which the operator is the multiplication by the independent variable.
	
	We start by noticing that \cite[Theorem 3.3]{KisNab} provides a description of the original Hilbert space $H,$ via its unitary equivalence to each of the two spaces 
	$K_{\mathfrak s},$ $K^\dagger_{\mathfrak s}.$  
	In our particular setup of extensions of minimal symmetric operators, this unitary equivalence is provided by the formulae (\ref{unit_form}).
	
	 We then use the representation of the inner product in $H$ in terms of the resolvent $R_\varkappa(z)$ via contour integration in the vicinity of the real line. Using the formulae (\ref{representation_minus})--(\ref{representation_plus}) and passing to the limit as the contour approaches a sum of integrals over the real line, we obtain one of the measures introduced in \cite{Clark, Aleksandrov_pseudocontinuable} (``Alexandrov-Clark measures") and subsequently studied in \cite{Poltora}; see also the survey \cite{Poltora_Sarason} and a recent development \cite{LMT}. In our context, this measure emerges from the Nevanlinna representation of the $M$-function ${\mathfrak m}.$ 
	 
	 We note that the resolvent representation provided by Theorem \ref{thm:rhyzhov2.5} can be shown to yield that $R_0(z)$ is the resolvent of a rank-one self-adjoint perturbation of a Toeplitz operator \cite{KisNab}, and thus the original argument of Clark \cite{Clark}, leading to the emergence of Aleksandrov-Clark measures and the functional model for the operator family $A_\varkappa,$ applies in our case. From this point of view, one can see the argument of the present section and Section \ref{sec_fin} as an independent proof of Clark's theorem, providing a straightforward and explicit formulae for the unitary operators mapping the original Hilbert space $H$ to the functional model. Although, for the reasons given above, the results to follow are not new on the abstract level, they yield an explicit functional model construction in terms of the objects naturally associated with the operator under consideration. 

	 

\subsection{Construction of a Clark-type measure for the model representation}
\label{Clark_constr}

Suppose that $\varkappa\in{\mathbb R}.$ For $\delta>0$ and $N\in{\mathbb R}_+$ that does not belong to the spectrum of the operator $A_\varkappa,$ denote by  $\Gamma_{\delta, N}$ 
the boundary of the rectangle 
\begin{equation}
	\{\zeta\in{\mathbb C}: |\Re\zeta|<N,\  |\Im\zeta|<\delta\}
	\label{rect}
\end{equation} 
and by $P_N$ the spectral projection for $A_\varkappa$ onto the interval $[-N, N].$ We also use the shorthand $u_N=P_Nu$ for all $u\in H.$ 

According to the Dunford-Riesz functional calculus \cite[Section XV.5]{DS}, one has, for all $\delta>0,$
\begin{equation}	
	\label{first-contour}
	P_N=-\frac 1{2\pi{\rm i}}
	\ointctrclockwise_{\Gamma_{\delta, N}} R_\varkappa (\lambda) d\lambda
	=\frac 1{2\pi{\rm i}}
	\ointclockwise_{\Gamma_{\delta, N}} R_\varkappa (\lambda) d\lambda,
\end{equation}
where $\Gamma_{\delta, N}$ is traced anticlockwise in the first integral in (\ref{first-contour})
and clockwise in the second integral in (\ref{first-contour}). Notice that $P_N\to I$ in the sense of strong operator convergence.
In what follows, we use
the notation
\begin{equation}
	\label{eq:lambda-expression}
	t:=\lambda-{\rm i}\delta=\Re\lambda\in \mathbb{R}.
\end{equation}
On the basis of (\ref{first-contour}), one has that the following 
analogue of the inverse Cauchy-Stieltjes formula:
\begin{equation}
	\label{eq:my-unbounded-cauchy}
\begin{aligned}
	\inner{u_N}{v_N}_H
	=\frac 1{2\pi\I}
	\int^{N}_{-N}
	\left\langle
	(R_\varkappa(\lambda)- R_\varkappa(\overline{\lambda}))u,v\right\rangle_Hdt+o(1),
	\qquad\forall u,v\in H,
\end{aligned}	
\end{equation}
where the term $o(1)$ goes to zero as $\delta\to0$ uniformly in $N.$


Assuming $(\widetilde{g}, g)^\top, (\widetilde{f}, f)^\top\in K$, we set $u=\Phi^{*}(\widetilde{g}, g)^\top$ and
$v=\Phi^{*}(\widetilde{g}, g)^\top$ and, using Theorem \ref{thm:rhyzhov2.5}
for $\lambda\in{\mathbb C}_+,$ write
\begin{gather}
\hspace{-3cm}	\left\langle (R_\varkappa(\lambda)- R_{\varkappa}(\overline{\lambda}))u,v\right\rangle_H
	=\left\langle P_K \frac 1{\cdot-\lambda}
	\binom{\widetilde{g}-\chi_\varkappa^-\widehat{\theta}^{-1}_\varkappa(\lambda)
		g_+(\lambda)}{g},
	\left(\begin{array}{c}\tilde{f}\\f\end{array}\right)\right\rangle_{\mathfrak H}\nonumber
	\\[0.4em]
	-\left\langle P_K \frac 1{\cdot-\overline{\lambda}}
	\left(\begin{array}{c}\tilde{g}\\g-\chi_{\varkappa}^+\theta_\varkappa^{-1}(\lambda)
		g_-(\lambda)\end{array}\right),
	\left(\begin{array}{c}\tilde{f}\\f\end{array}\right)\right\rangle_{\mathfrak H}
	\nonumber
	\\[0.4em]
	\hspace{+6cm}=\left\langle \left(\frac 1{\cdot-\lambda}-\frac
	1{\cdot-\overline{\lambda}}\right)\left(\begin{array}{c}\tilde{g}\\g\end{array}\right)
	\left(\begin{array}{c}\tilde{f}\\f\end{array}\right)\right\rangle_{\mathfrak H}
	-\mathcal{G}(\lambda,\overline{\lambda}),
	\label{my-last-equality}
\end{gather}
where
\begin{equation}
	\begin{aligned}
	\mathcal{G}(\lambda,\overline{\lambda})&:=\int_{\mathbb R}\frac{1}{k-\lambda}\chi_\varkappa^-\widehat{\theta}_{\varkappa}^{-1}(\lambda)g_+(\lambda)
	\overline{f_-(k)}dk
	-\int_{\mathbb R}\frac 1{k-\overline{\lambda}}\chi_\varkappa^+\theta_{\varkappa}^{-1}(\overline{\lambda})g_-(\overline{\lambda})\overline{f_+(k)} dk\\[0.4em]
	&=2\pi {\rm i}\bigl\{\chi_\varkappa^-\widehat\theta_{\varkappa}^{-1}(\lambda) g_+ (\lambda)\overline{f_-(\overline{\lambda})}+
	\chi_\varkappa^+\theta_{\varkappa}^{-1}(\overline{\lambda}) g_-(\overline{\lambda})\overline{f_+(\lambda)}\bigr\}.
\end{aligned}
\label{gcalc1}
\end{equation}
For the second equality in (\ref{gcalc1}) we have used the fact that for $h\in H^\pm_2$
the following identities hold:
\begin{equation}
	\int_\mathbb{R} \frac{h(k)}{k-z}dk=\pm 2\pi{\rm i} h(z),
	\qquad z \in \mathbb{C}^\pm.
	\label{hid}
\end{equation}

Taking into account
\eqref{eq:lambda-expression}, the first term on the right-hand side of 
\eqref{my-last-equality} can be written as follows:
\begin{equation}
	\begin{aligned}
	\left\langle \left(\frac 1{\cdot-\lambda}-\frac
	1{\cdot-\overline{\lambda}}\right)\left(\begin{array}{c}\tilde{g}\\g\end{array}\right),
	\left(\begin{array}{c}\tilde{f}\\f\end{array}\right)\right\rangle_{\mathfrak H}
	=
	\int_{-\infty}^{\infty}\frac
	{2{\rm i}\delta}{(k-t)^2+\delta^2}\mathcal{F}(k)dk,
	\end{aligned}
	\label{first_par}
\end{equation}
where
\begin{equation*}
	\mathcal{F}(k):=\left\langle
	\left(\begin{array}{cc}1&\cc{\mathfrak{s}(k)}\\
		\mathfrak s(k)&1\end{array}\right)\left(\begin{array}{c}\tilde g(k)\\ g(k)\end{array}\right),
	\left(\begin{array}{c}\tilde f(k)\\ f(k)\end{array}\right)\right\rangle_{{\mathbb C}^2}.
\end{equation*}

Now, integrating (\ref{first_par}) with respect to $t\in[-N, N]$ (where $t$ and $\lambda$ are related via \eqref{eq:lambda-expression}),
one obtains 
\begin{equation}
	\label{eq:my-first-parenthesis}
	\begin{split}
		\int_{-N}^{N}\biggl\langle
		&\left(\frac{1}{\cdot-\lambda}-\frac{1}{\cdot-\overline{\lambda}}\right)
		\left(\begin{array}{c}\tilde{g}\\ g\end{array}\right),
		\left(\begin{array}{c}\tilde{f}\\ f\end{array}\right)
		\biggr\rangle_{\mathfrak H}dt
		=\frac 1{2\pi \I}
		\int_{-N}^{N}\int_{-\infty}^{\infty} 
		\frac {2{\rm i}\delta}{(k-t)^2+\delta^2}\mathcal{F}(k)dk dt\\[0.4em]
		&=
		\int\limits_{-N}^{N}
		\mathcal{F}(k)
		\int\limits_{-N}^{N}
		\frac \delta\pi
		\frac1{(k-t)^2+\delta^2}dtdk+o(1)
		=\int\limits_{-N}^{N}g_{+}(k)\cc{f_{+}(k)}dk+o(1).
	\end{split}
\end{equation}
In view of \eqref{my-last-equality}, we rewrite
\eqref{eq:my-unbounded-cauchy} by substituting \eqref{gcalc1} and
\eqref{eq:my-first-parenthesis} into it:
\begin{equation*}
	\inner{u_N}{v_N}_H=\int\limits_{-N}^{N}g_{+}(k)\cc{f_{+}(k)}dk-
	\int\limits_{-N}^{N}
	\left( \chi_{\varkappa}^{-}\widehat\theta_{\varkappa}^{-1}(\lambda)
	g_{+}(\lambda)\cc{f_{-}(\overline{\lambda})}+
	\chi_{\varkappa}^+\theta_{\varkappa}^{-1}(\overline{\lambda})
	g_{-}(\cc{\lambda})\cc{f_{+}(\lambda)}\right)dt+o(1).
\end{equation*}

Consider a region $\Omega\subset{\mathbb C}$ containing the real line that has no poles or zeros of ${\mathfrak s}.$ 
This is possible due to the fact that $A_{\rm min}$ is simple. Indeed, the operator $A_{\rm i}$ is completely non-selfadjoint and dissipative, which prevents it from having real eigenvalues. This, in turn, ensures that the zeros (and hence the poles as well) of ${\mathfrak s}$ are also away from the real line, as they coincide with the spectrum of $A_{\rm i}$ (and its adjoint, respectively).

Furthermore, for each $N$ as above, choose $\delta_N$ so that
 $\Gamma_{\delta, N}\subset\Omega$ for all $\delta<\delta_N,$ where $\delta$ is defined in (\ref{rect}) and is related to $\lambda$ via (\ref{eq:lambda-expression}).
In the context of the present paper, we are interested in the member of the family $A_\varkappa$ that corresponds to the value $\varkappa=0,$ which we assume to be selected from now on. However, the argument to follow can be extended to also work with other values $\varkappa\in{\mathbb R}.$

Aiming at the operator ${\mathcal A}^\tau_{\rm hom},$ $\tau\in[-\pi, \pi),$ introduced in Section \ref{section:examples}, in the remainder of this section we set $\varkappa=0.$ Taking into account \eqref{eq:chi-def}
and \eqref{eq:alternate-for-theta-hat},
one obtains
\begin{equation}
\begin{aligned}
	\inner{u_N}{v_N}_H&=\int\limits_{-N}^{N}g_{+}(k)\cc{f_{+}(k)}dk-
	\int\limits_{-N}^{N}
	\left(\frac{g_{+}(\lambda)\cc{f_{-}(\overline{\lambda})}}{1+\mathfrak
		s(\lambda)}+\frac
	{g_{-}(\cc{\lambda})\cc{f_{+}(\lambda)}}{1+\cc{\mathfrak s(\lambda)}}
	\right)dt+o(1)\\
	&=\int\limits_{-N}^{N}g_{+}(k)\cc{f_{+}(k)}dk-
	\int\limits_{-N}^{N}
	\left(\frac{\mathfrak s(\lambda)g_{+}(\lambda)\cc{f_{+}(\lambda)}}{1+\mathfrak
		s(\lambda)}+\frac
	{\cc{\mathfrak s(\lambda)}g_{+}(\lambda)\cc{f_{+}(\lambda)}}{1+\cc{\mathfrak s(\lambda)}}
	\right)dt+o(1),
\end{aligned}
\label{uv_id}
\end{equation}
where for the last equality we have used the identities
\begin{equation}
	h_-(z)=\overline{{\mathfrak s}(\overline{z})}h_+(\overline{z}),\quad z\in{\mathbb C}_-,\qquad\qquad
	h_+(z)={\mathfrak s}(z)h_-(\overline{z}),\quad z\in{\mathbb C}_+,
	\label{fg_id}
\end{equation}
obtained by analytic continuation into $\Omega.$
Furthermore, noticing that
\begin{equation}
{\mathfrak s}(1+{\mathfrak s})^{-1}
=1-(1+{\mathfrak s})^{-1},
\label{Sid}
\end{equation}
we rewrite (\ref{uv_id}) as follows:
\begin{align*}
		\langle u_N, v_N\rangle_H&-\int_{-N}^N g_+ (t)
			\overline{f_+(t)}dt\\
		&=
		-\int\limits_{-N}^{N}\biggl\{g_+ (\lambda)
		\overline{f_+(\lambda)}-\frac{1}{1+{\mathfrak s}(\lambda)} g_+ (\lambda)
		\overline{f_+(\lambda)}
		+\frac{\overline{{\mathfrak s}(\lambda)}}{1+\overline{{\mathfrak s}(\lambda)}}g_+ (\lambda)\overline{f_+(\lambda)}\biggr\}dt+o(1)\\
		&=-\int_N^N g_+ (t)
			\overline{f_+(t)}dt
		+\int\limits_{-N}^{N}\biggl\{\frac{1}{1+{\mathfrak s}(\lambda)}-\frac{\overline{{\mathfrak s}(\lambda)}}{1+\overline{{\mathfrak s}(\lambda)}}\biggr\}g_+(t)\overline{f_+(t)}dt+o(1).
\end{align*}
Using the identity 
\begin{equation}
\frac{\overline{{\mathfrak s}(\lambda)}}{1+\overline{{\mathfrak s}(\lambda)}}=\frac{1}{\overline{{\mathfrak s}(\lambda)}^{-1}+1}=\frac{1}{{\mathfrak s}(\overline{\lambda})+1}, \qquad \lambda\in{\mathbb C}_+,
\label{another_sid}
\end{equation}
we therefore have
\begin{equation}
	\langle u_N, v_N\rangle_H=
	\int\limits_{-N}^{N}\biggl\{\frac{1}{1+{\mathfrak s}(\lambda)}-\frac{1}{1+{\mathfrak s}(\overline{\lambda})}\biggr\}
	g_+(t)\overline{f_+(t)}dt.
	\label{pre_eps}
\end{equation}

Finally, we combine this with the representation
\begin{equation}
\frac{1}{1+{\mathfrak s}(\lambda)}=\frac{1}{2}\biggl(1+\frac{\rm i}{{\mathfrak m}(\lambda)}\biggr)=C_0+C_1\lambda-\frac{\rm i}{2}\int_{\mathbb R}\biggl(\frac{1}{\sigma-\lambda}-\frac{\sigma}{1+\sigma^2}\biggr)d\mu(\sigma),
\label{Nev_rep}
\end{equation}
where $C_0, C_1\in{\mathbb C},$ and $\mu$ is the measure of the Nevanlinna representation of the Herglotz function $-{\mathfrak m}^{-1},$ see {\it e.g.} \cite[Section 5.3]{MR822228}. The formula (\ref{Nev_rep}) implies, in particular, that
\begin{equation}
\frac{1}{1+\mathfrak s(\lambda)}-\frac{1}{1+
	\mathfrak s(\overline{\lambda})}=-\frac{\rm i}{2}\int\limits_{\mathbb{R}}\frac{2{\rm i} \delta d\mu(\sigma)}
{(\sigma-t)^2+\delta^2}=
\pi\mathfrak{P}_\delta
(\mu)(t),
\label{still_rep}
\end{equation}
where $\mathfrak{P}_\delta$ stands for the Poisson transformation. Next, note that $\mu$ is a Clark measure \cite{Aleksandrov_pseudocontinuable}, due to (\ref{still_rep}) and the identity
\[
\frac{1}{1+\mathfrak s(\lambda)}-\frac{1}{1+
	\mathfrak s(\overline{\lambda})}=\Re\frac{1-{\mathfrak s}(\lambda)}{1+{\mathfrak s}(\lambda)},\qquad \lambda\in{\mathbb C}_+,
\]
obtained directly from (\ref{another_sid}).

Substituting (\ref{still_rep}) into (\ref{pre_eps})  and taking into
account the weak*-convergence 
\cite[VI Sec.\,B]{MR1669574} of the Poisson transformations
 ${\mathfrak P}_\delta$ as well as the regularity \cite{Nikolski} of functions in $K_{\mathfrak s}$ guaranteed by the analytic properties of ${\mathfrak s}$ discussed above,  we pass to the limit as $\delta\to0$ in (\ref{pre_eps}), to obtain 
\begin{equation*}
\langle u_N, v_N\rangle=\pi\int_{-N}^N g_+ (t)\overline{f_+(t)}d\mu(t)+o(1).
\end{equation*}
Finally, passing to the limit in the last identity as $N\to\infty$ and using the fact that $u_N\to u,$ $v_N\to v$  yields 
\begin{equation*}
	\langle u, v\rangle=\pi\int_{-\infty}^\infty g_+ (t)\overline{f_+(t)}d\mu(t),\qquad u, v\in H.
\end{equation*}

We have thus established the following theorem.
\begin{theorem}
	\label{first_thm}
	The Hilbert space $H$ is isometric to the space 
	$L_2\bigl(\mathbb{R}, \pi d\mu\bigr),$ where the measure $\mu$ is provided by (\ref{still_rep}). This  isometry is the composition of the first formula in (\ref{unit_form}) and the embedding of $K_{\mathfrak s}$ into $L_2\bigl(\mathbb{R}, \pi d\mu\bigr)$ realised by taking the boundary values on the real line of functions in $K_{\mathfrak s},$ which exist $\mu$-almost everywhere.
\end{theorem}





\begin{remark}
A. Unlike in \cite{Clark}, here the Clark measure $d\mu$ emerges in the context of extensions of symmetric operators, via the operators of the functional model. 

B. Theorem \ref{first_thm} admits an alternative proof by combining classical results by Clark \cite{Clark}, concerning the isometry between $K_{\mathfrak s}$ and $L_2\bigl(\mathbb{R}, \pi d\mu\bigr),$ and by Poltoratski \cite{Poltora} (see also the survey \cite{Poltora_Sarason}), concerning the realisation of the mentioned isometry via passing to the boundary values on the real line.
\end{remark}


\subsection{The resolvent as an operator of multiplication by the independent variable}

Fix $\varkappa\in{\mathbb R},$ $z\in{\mathbb C}^+\cup{\mathbb C}^-,$ and consider $\delta\in(0, |\Im z|)$ and $N\in{\mathbb R}$ as described at the beginning of Section \ref{Clark_constr}.  Similarly to the above, we write
\begin{equation}
	R_\varkappa(z)P_N=-\frac 1{2\pi{\rm i}}
	\ointctrclockwise_{\Gamma_{\delta, N}} (z-\lambda)^{-1}R_\varkappa (\lambda) d\lambda
	=\frac 1{2\pi{\rm i}}
	\ointclockwise_{\Gamma_{\delta, N}}(z-\lambda)^{-1}R_\varkappa (\lambda) d\lambda,
	\label{contour}
\end{equation}
where $\Gamma_{\delta, N}$ is the boundary of the rectangle (\ref{rect}) and the integrals are understood in the same sense as (\ref{first-contour}).
Using (\ref{contour}), we can write 
\begin{equation}
	\label{eq:unbounded-cauchy}
	\bigl\langle R_\varkappa(z)u_N, v_N\bigr\rangle_H=\frac 1{2\pi {\rm i}}\int\limits^{N}_{-N}\bigl\langle
	\bigl\{(z-\lambda)^{-1}R_\varkappa(\lambda)-
	(z-\overline{\lambda})^{-1}R_\varkappa(\overline{\lambda})\bigr\}u,v
	\bigr\rangle_H dt
	\qquad\forall u,v\in H.
\end{equation}
Assuming $(\tilde{g}, g)^\top,$  $(\tilde{f}, f)^\top\in K$,
let $u=\Phi^{*}(\tilde{g}, g)^\top$ and
$v=\Phi^{*}(\tilde{f}, f)^\top.$ Then one has ({\it cf.} (\ref{my-last-equality}))
\begin{align}
	\bigl\langle\bigl\{	(z&-\lambda)^{-1}R_\varkappa(\lambda)- 	(z-\overline{\lambda})^{-1}R_{\varkappa}(\overline{\lambda})\bigr\}u,v\bigr\rangle_H\nonumber\\[0.5em]
	&=\left\langle\frac{1}{z-\lambda}\,\frac 1{\cdot-\lambda}
	\left(\begin{array}{c}\tilde{g}-\chi_{\varkappa}^-\widehat\theta_\varkappa^{-1}(z)({\mathfrak s}\tilde g+g)(z)\\g\end{array}\right),
	\left(\begin{array}{c}\tilde{f}\\f\end{array}\right)\right\rangle_{\mathfrak H}
	\nonumber
	\\[0.4em]
	&-\left\langle\frac{1}{z-\overline{\lambda}}\,\frac 1{\cdot-\overline{\lambda}}
	\left(\begin{array}{c}\tilde{g}\\g-\chi_{\varkappa}^*\theta_\varkappa^{-1}(z)(\tilde{g}+\overline{\mathfrak s}g)(z)\end{array}\right),
	\left(\begin{array}{c}\tilde{f}\\f\end{array}\right)\right\rangle_{\mathfrak H}
	\nonumber
	\\[0.4em]
	&
	=\left\langle\left(\frac{1}{z-\lambda}\,\frac 1{\cdot-\lambda}-\frac{1}{z-\overline{\lambda}}\,\frac
	1{\cdot-\overline{\lambda}}\right)\left(\begin{array}{c}\tilde{g}\\g\end{array}\right),
	\left(\begin{array}{c}\tilde{f}\\f\end{array}\right)\right\rangle_{\mathfrak H}
	-\widetilde{\mathcal{G}}(\lambda,\overline{\lambda}),
	\label{last-equality}
\end{align}
where
\begin{equation}
\begin{aligned}
	\widetilde{\mathcal{G}}(\lambda,\overline{\lambda})
	&:=\frac{1}{z-\lambda}\int_{\mathbb R}\frac{1}{k-\lambda}\chi_\varkappa^-\widehat{\theta}_{\varkappa}^{-1}(\lambda)g_+(\lambda)
	\overline{f_-(k)}dk
	-\frac{1}{z-\overline{\lambda}}\int_{\mathbb R}\frac 1{k-\overline{\lambda}}\chi_\varkappa^+\theta_{\varkappa}^{-1}(\overline{\lambda})g_-(\overline{\lambda})\overline{f_+(k)} dk
	\\[0.45em]
	&=2\pi{\rm i}\biggl\{\frac{1}{z-\lambda}\chi_\varkappa^-\widehat\theta_{\varkappa}^{-1}(\lambda)g_+(\lambda)
		\overline{f_-(\overline{\lambda})}
		+\frac{1}{z-\overline{\lambda}} \chi_\varkappa^+\theta_{\varkappa}^{-1}(\overline{\lambda})g_-(\overline{\lambda})
		\overline{f_+(\lambda)}\biggr\}.
	\end{aligned}
\label{gcalc}
\end{equation}
Here, for the first equality we have used the identities (\ref{hid}).
The first term in
\eqref{last-equality} can be re-written as follows:
\begin{equation}
\begin{aligned}
	\biggl\langle\biggl(\frac{1}{z-\lambda}\,\frac 1{\cdot-\lambda}&-\frac{1}{z-\overline{\lambda}}\,\frac
	1{\cdot-\overline{\lambda}}\biggr)\left(\begin{array}{c}\tilde{g}\\g\end{array}\right),
	\left(\begin{array}{c}\tilde{f}\\f\end{array}\right)\biggr\rangle_{\mathfrak H}\\[0.4em]  
	&=\int_{-\infty}^{\infty}\frac{1}{k-z}\biggl\{\frac
	{2{\rm i}\delta}{(k-t)^2+\delta^2}-
	\frac
	{2{\rm i}\delta}{(z-t)^2+\delta^2}\biggr\}\\[0.4em]
	&\hspace{3cm}\times\left\langle
	\left(\begin{array}{cc}1&\overline{{\mathfrak s}(k)}\\
		{\mathfrak s}(k)&1\end{array}\right)\left(\begin{array}{c}\tilde g(k)\\ g(k)\end{array}\right),
	\left(\begin{array}{c}\tilde f(k)\\ f(k)\end{array}\right)\right\rangle_{{\mathbb C}^2}dk.	
\end{aligned}
\label{eq:first-parenthesis}
\end{equation}
Similarly to the calculation (\ref{eq:my-first-parenthesis}), for the integral of the expression (\ref{eq:first-parenthesis}) with respect to $t\in[-N, N]$ we obtain 
\begin{align*}
\int_{-N}^N\biggl\langle\biggl(\frac{1}{z-\lambda}\,\frac 1{\cdot-\lambda}&-\frac{1}{z-\overline{\lambda}}\,\frac
1{\cdot-\overline{\lambda}}\biggr)\left(\begin{array}{c}\tilde{g}\\g\end{array}\right),
\left(\begin{array}{c}\tilde{f}\\f\end{array}\right)
\biggr\rangle_{\mathfrak H}dt\\[0.4em]
&=2\pi{\rm i}\int_{-N}^{N}\frac{1}{k-z}\left\langle
\left(\begin{array}{cc}1&\overline{{\mathfrak s}(k)}\\
	{\mathfrak s}(k)&1\end{array}\right)\left(\begin{array}{c}\tilde g(k)\\ g(k)\end{array}\right),
\left(\begin{array}{c}\tilde f(k)\\ f(k)\end{array}\right)\right\rangle_{{\mathbb C}^2}dk+o(1)\\[0.4em]
&=-2\pi{\rm i}\int_{-N}^N\frac{g_+ (t)
	\overline{f_+(t)}}{z-t}dt+o(1).
\end{align*}
 Combining this with (\ref{eq:unbounded-cauchy}), (\ref{last-equality}), and (\ref{gcalc}), we obtain
\begin{align*}
		\bigl\langle R_\varkappa(z)u_N, v_N\bigr\rangle_H&+\int_{-N}^N\frac{g_+ (t)
			\overline{f_+(t)}}{z-t}dt=-\frac 1{2\pi
		{\rm i}}
	\int\limits_{-N}^{N}\widetilde{\mathcal{G}}(\lambda,\overline{\lambda})dt+o(1)\\
	    &=
	    \int\limits_{-N}^{N}\biggl\{\frac{1}{z-\lambda}\chi_\varkappa^-\widehat\theta_{\varkappa}^{-1}(\lambda) g_+ (\lambda)\overline{f_-(\overline{\lambda})}+\frac{1}{z-\overline{\lambda}}
	\chi_\varkappa^+\theta_{\varkappa}^{-1}(\overline{\lambda}) g_-(\overline{\lambda})\overline{f_+(\lambda)}\Bigr\}dt+o(1).
	\end{align*}
Setting $\varkappa=0$ (which corresponds to the operator ${\mathcal A}^\tau_{\rm hom},$ $\tau\in[-\pi,\pi),$ introduced in Section \ref{section:examples}) and using the identities (\ref{fg_id}) yields	
\begin{equation}
\begin{aligned}	
\bigl\langle R_0(z)u_N, v_N\bigr\rangle_H&+\int_{-N}^N\frac{g_+ (t)
	\overline{f_+(t)}}{z-t}dt\\[0.4em]
&=
\int\limits_{-N}^{N}\biggl\{\frac{1}{z-\lambda}\frac{{\mathfrak s}(\lambda)}{1+{\mathfrak s}(\lambda)} g_+ (\lambda)
\overline{f_+(\lambda)}
+\frac{1}{z-\overline{\lambda}}\frac{\overline{{\mathfrak s}(\lambda)}}{1+\overline{{\mathfrak s}(\lambda)}}g_+ (\lambda)\overline{f_+(\lambda)}\biggr\}dt+o(1).
\end{aligned}
\label{Rlim}
\end{equation}

Using the identity (\ref{Sid}), we rewrite (\ref{Rlim}) as follows:
\begin{equation*}
	\begin{aligned}	
		&\bigl\langle R_0(z)u_N, v_N\bigr\rangle_H+\int_{-N}^N\frac{g_+ (t)
			\overline{f_+(t)}}{z-t}dt\\
		&=
		\int\limits_{-N}^{N}\biggl\{\frac{1}{z-\lambda}g_+ (\lambda)
		\overline{f_+(\lambda)}-\frac{1}{z-\lambda}\frac{1}{1+{\mathfrak s}(\lambda)} g_+ (\lambda)
		\overline{f_+(\lambda)}
		+\frac{1}{z-\overline{\lambda}}\frac{\overline{{\mathfrak s}(\lambda)}}{1+\overline{{\mathfrak s}(\lambda)}}g_+ (\lambda)\overline{f_+(\lambda)}\biggr\}dt+o(1)\\
		&=\int_{-N}^N\frac{g_+ (t)
		\overline{f_+(t)}}{z-t}dt
	   -\int\limits_{-N}^{N}\frac{1}{z-t}\biggl\{\frac{1}{1+{\mathfrak s}(\lambda)}-\frac{\overline{{\mathfrak s}(\lambda)}}{1+\overline{{\mathfrak s}(\lambda)}}\biggr\}g_+(t)\overline{f_+(t)}dt+o(1).
	\end{aligned}
\end{equation*}


Choosing, for each $N,$ a value $\delta_N<|\Im z|$ such that 
the rectangle (cf. (\ref{rect}))
\begin{equation*}
	\{\zeta\in{\mathbb C}: |\Re\zeta|<N,\  |\Im\zeta|<\delta_N\}
\end{equation*} 
contains no poles or zeros of ${\mathfrak s}$  and using the identity (\ref{another_sid}), we therefore have, for all $\delta<\delta_N,$
\begin{align*}
	\bigl\langle R_0(z)u_N, v_N\bigr\rangle_H=
	\int\limits_{-N}^{N}
	\frac{1}{t-z}\biggl\{\frac{1}{1+{\mathfrak s}(\lambda)}-\frac{1}{1+{\mathfrak s}(\overline{\lambda})}\biggr\}g_+(t)\overline{f_+(t)}dt+o(1).
\end{align*}
Combining this with the representation (\ref{still_rep}) and passing to the limit as $\delta\to0$ yields
\[
\bigl\langle R_0(z)u_N, v_N\bigr\rangle_H=\pi\int_{-N}^N\frac{g_+ (t)
	\overline{f_+(t)}}{t-z}d\mu(t)+o(1).
\]
Finally, passing to the limit as $N\to\infty,$ we obtain 
\[
\bigl\langle R_0(z)u, v\bigr\rangle_H=\pi\int_{-\infty}^\infty \frac{g_+ (t)
	\overline{f_+(t)}}{t-z}d\mu(t).
\]

	We have thus established the following theorem.
	\begin{theorem}
	 Under the isometry described in Theorem \ref{first_thm}, the resolvent $(A_0-z)^{-1}$ is unitarily equivalent to the operator of
		multiplication by $(\cdot-z)^{-1}$ in the space $L_2(\mathbb{R}, \pi d\mu)$. 
	\end{theorem}


\section{Application to high-contrast homogenisation: an explicit functional model representation}
\label{sec_fin}

Substituting the expression (\ref{Mfunction_expl}) into (\ref{SviaM}) and using the Stieltjes inversion formula, see, e.g., \cite[p.\,9]{AG}, \cite[Section 5.4]{MR822228}, we infer that $\mu$ is a counting measure with masses located at the poles $\lambda=\lambda_j,$ $j=1,2,\dots,$ of the expression ($\lambda\in{\mathbb R}_+$)
\begin{equation}
	\begin{aligned}
\frac{1}{1+{\mathfrak s}(\lambda)}=\frac{1}{2}\biggl(1+\frac{\rm i}{{\mathfrak m}(\lambda)}\biggr)&=\frac{1}{2}-\frac{{\rm i}\sqrt{\lambda}\bigl(\Re({\rm e}^{{\rm i}\tau l}\overline{\omega})-\cos\sqrt{\lambda}l\bigr)(\eta^2\lambda-\gamma)}{(\eta^2\lambda-\gamma)\sin\sqrt{\lambda}l+2\sqrt{\lambda}\bigl(\Re({\rm e}^{{\rm i}\tau l}\overline{\omega})-\cos\sqrt{\lambda}l\bigr)}\\
&=C_0+C_1\lambda-\frac{\rm i}{2}\int\limits_{\mathbb{R}} \frac{d\mu(\sigma)}{\lambda-\sigma},
\end{aligned}
\label{denom}
\end{equation}
where $C_0, C_1$ are defined via (\ref{Nev_rep}). Clearly, these solve the transcendental equation for $z=\lambda_j$ obtained by setting to zero the denominator in (\ref{denom}):
\begin{equation}
\cos\sqrt{\lambda_j}l-(\eta^2\lambda_j-\gamma)\frac{\sin\sqrt{\lambda_j}l}{2\sqrt{\lambda_j}}=\Re\bigl({\rm e}^{{\rm i}\tau l}\overline{\omega}\bigr),\qquad j=1,2,\dots
\label{dispersion_fin}
\end{equation}
The corresponding mass 
is given by evaluating the residue of the expression (\ref{denom}) at the pole $\lambda_j:$
\[
\mu(\{\lambda_j\})=\frac{2(\eta^2\lambda_j-\gamma)^2}{\eta^2+\dfrac{\gamma}{\lambda_j}+2l+\dfrac{l}{\sqrt{\lambda_j}}(\eta^2\lambda_j-\gamma)\cot\sqrt{\lambda_j}l}.
\]

Using the values (\ref{parameters}), one immediately obtains a representation for the resolvent $({\mathcal A}^\tau_{\rm hom}-z)^{-1}$ of the operator ${\mathcal A}^\tau_{\rm hom}$ introduced in Section \ref{section:examples} as the operator of multiplication by $(\cdot-z)^{-1}$ in $L^2(\mathbb{R}, \pi d\mu).$ 
In this context the measure $\mu$ is parametrised by $\varepsilon$ and $\tau\in[-\pi,\pi).$ In fact, it shows a ``two-scale" dependence on the quasimomentum, being a function  $\tau$ and $\tau/\varepsilon$ only: the equation (\ref{dispersion_fin}) reads
\[
\cos\sqrt{\lambda_j}l_2-\biggl\{(l_1+l_3)\lambda_j-\biggl(\frac{l_1}{a_1}+\frac{l_3}{a_3}\biggr)^{-1}\biggl(\frac{\tau}{\varepsilon}\biggr)^2\biggr\}\frac{\sin\sqrt{\lambda_j}l_2}{2\sqrt{\lambda_j}}
=\frac{\dfrac{a_1}{l_1}\cos\tau+\dfrac{a_3}{l_3}}
{\sqrt{\biggl(\dfrac{a_1}{l_1}\cos\tau+\dfrac{a_3}{l_3}\biggr)^2+\dfrac{a_1^2}{l_1^2}\sin^2\tau}}.
\]
where we have used the assumption that $l_1+l_2+l_3=1.$ In the particular case when $a_1=a_3=a,$ $l_1=l_3=(1-l_2)/2,$ it takes a more compact form, as follows:
\[
\cos\sqrt{\lambda_j}l_2-\biggl\{(1-l_2)\lambda_j-\frac{a}{1-l_2}\biggl(\frac{\tau}{\varepsilon}\biggr)^2\biggr\}\frac{\sin\sqrt{\lambda_j}l_2}{2\sqrt{\lambda_j}}=|\cos\tau|, \qquad j=1,2,\dots
\]

Apart from the usual implications of an explicit functional model representation thus constructed on the spectral analysis of the operator ${\mathcal A}^\tau_{\rm hom}$, we have obtained a special (``spectral") representation for the generalised resolvent (in the form of an explicit pseudodifferential operator) 
\[
{\mathcal R}^\tau_{\rm hom}(z)=\mathcal{P}({\mathcal A}^\tau_{\rm hom}-z)^{-1}\big|_{L^2(0, l_2)},
\]
for which the operator ${\mathcal A}_{\rm hom}^\tau$  serves as the Neumark-\v{S}trauss dilation. Here $\mathcal{P}$ is the natural orthogonal projection of $L^2(0, l_2)\oplus{\mathbb C}$ onto  $L^2(0, l_2).$ 

When considered in the context of ``multipole" homogenisation  representations, this will allow us to demonstrate ``metamaterial" properties, in particular antiparallel group and phase velocities. These multipole representations will of course require that one passes from the ``scalar" context (where the key objects involved, i.e., the $M$-function $\mathfrak m$ and the characteristic function $\mathfrak{s}$) to a ``matrix" one. The details of the related argument will appear in forthcoming publication.


\section*{Acknowledgements}


KDC is grateful for the financial support of
EPSRC Grants EP/L018802/2, EP/V013025/1. YYE and SNN acknowledge financial support by the Russian Science Foundation Grant No.\,20-11-20032. KDC and YYE have been partially supported by CONACyT CF-2019 No.\,304005.

We are grateful to Dr A.\,V.\,Kiselev for reading the paper and providing a number of insightful comments. 


\end{document}